\documentclass[final]{siamart1116}

\usepackage{bbm}
\usepackage{amsmath}
\usepackage{amssymb}
\usepackage{graphicx}
\usepackage{subfigure}
\usepackage{color}
\usepackage{algorithm}
\usepackage{algorithmic}
\usepackage{multirow}

\def\R{{\Bbb R}}

\def\Om{\Omega}

\def\ep{\epsilon}

\def\f{\frac}
\def\p{\partial}

\def\na{\nabla}

\def\e{{\bf e}}

\def\h{{\mathbf h}}
\def\n{\mathbf{n}}

\def\r{\mathbf{r}}
\def\x{\mathbf{x}}
\def\y{\mathbf{y}}

\def\V{\mathbf V}

\def\mE{{\mathcal E}}

\def\nax1x2{\na_{x_1x_2}}
\def\nay1y2{\na_{y_1y_2}}
\def\nax{\na_{\tiny\x}}
\def\nay{\na_{\tiny\y}}

\def\eref#1{(\ref{#1})}

\def\cor#1{{{#1}}}

\def\feit{f_{\mbox{\tiny EIT}}}
\def\fVh{f_{\mbox{\tiny Vh}}}

\title{A learning-based method for solving ill-posed nonlinear inverse problems: a simulation study of Lung EIT}
\author{Jin Keun Seo\thanks{Department of Computational Science and Engineering, Yonsei University, Korea (\email{seoj@yonsei.ac.kr}, \email{kangcheol@yonsei.ac.kr}, \email{j.ariungerel@gmail.com}, \email{imlkh84@gmail.com})}
\and Kang~Cheol~Kim\footnotemark[1]
\and Ariungerel Jargal\footnotemark[1]
\and Kyounghun~Lee\footnotemark[1]
\and Bastian Harrach\thanks{Department of Mathematics, Goethe University Frankfurt, Germany (\email{harrach@math.uni-frankfurt.de})}}

\begin{document}
       \maketitle
       
\let\thefootnote\relax\footnotetext{\hrule \vspace{1ex} \centering This is a preprint version of a journal article published in\\
 \emph{SIAM J. Imaging Sci.} \textbf{12}(3), 1275--1295, 2019
(\url{https://doi.org/10.1137/18M1222600}).
}

\begin{abstract}
This paper proposes a new approach for solving ill-posed nonlinear inverse problems. For ease of explanation of the proposed approach, we use the example of lung electrical impedance tomography (EIT), which is  known to be a nonlinear and ill-posed inverse problem. Conventionally, penalty-based regularization methods have been used to deal with the ill-posed problem. However, experiences over the last three decades have shown methodological limitations in utilizing prior knowledge about tracking expected imaging features for medical diagnosis. The proposed method's paradigm is completely different from conventional approaches; the proposed reconstruction uses a variety of training data sets to generate a low dimensional manifold of approximate solutions, which allows to convert the ill-posed problem to a well-posed one.  Variational autoencoder was used to produce a compact and dense representation for lung EIT images with a low dimensional latent space. Then, we learn a robust connection between the EIT data and the low-dimensional latent data. Numerical simulations validate the effectiveness and feasibility of the proposed approach. \end{abstract}

\section{Introduction}
Electrical impedance tomography (EIT) aims to provide tomographic images of an  electrical conductivity distribution inside an electrically conducting object such as the human body \cite{Barber1984,Brown1985,Barber1989,Henderson1978,Holder2005,Seo2013}. In EIT, we attach an array of surface electrodes  around a chosen imaging slice of the object to inject currents and measure the induced voltages. Noting that current-voltage relation reflects the conductivity distribution  according to Ohm's law, an accurate conductivity reconstruction by EIT is theoretically possible \cite{Astala2006,Calderon1980,Kenig2007,Kohn1984,Nachman1988,Nachman1996,Sylvester1987}.

However, EIT in a clinical setting has suffered from the fundamental limitations that current-voltage data is very sensitive to the forward modeling errors involving the boundary geometry and the electrode configuration, whereas it is insensitive to local perturbation of the conductivity. Since the inverse problem of EIT is highly ill-posed, the most common techniques are regularized model-fitting approaches (e.g. least square minimization combined with regularization) \cite{Cheney1999,Lionheart2004}. Unfortunately, in a clinical environment, these techniques have not provided satisfactory results in terms of accuracy and resolution, despite of numerous endeavours in the last four decades. Within conventional regularization frameworks including Tikhonov \cite{Tikhonov1977} and total variation regularization, it seems to be very difficult to enforce prior knowledge of possible solutions effectively.

This paper suggests a new paradigm of EIT reconstruction using a specially designed deep learning framework to leverage prior knowledge of solutions. For ease of explanations, this paper focuses on the mathematical model of the time-difference EIT imaging of air ventilation in the lungs. We denote by $\gamma_{t}(\r)$ the conductivity at time $t$ and position $\r$.
The input data for the deep learning is the time-difference of the current-voltage data $\dot{\V}_{t}:= \V_{t}-\V_{t_0}$ in EIT (see section 2 for $\dot{\V}$)  and the output is the difference conductivity image $\dot{\gamma}_{t}:=\gamma_{t}-\gamma_{t_0}$, where $t_0$ denotes a reference time. With fixing time $t$, we will use the shorter notations $\dot{\gamma}$ and $\dot{\V}$ instead of $\dot{\gamma}_{t}$ and $\dot{\V}_{t}$, respectively.   The goal is to learn an EIT reconstruction map $\feit$ from training data set $\{ (\dot{\V}^n, \dot{\gamma}^n): n=1, \ldots, N\}$ such that $\feit (\dot{\V})$ produces a useful reconstruction for $\dot{\gamma}$.

The standard deep learning paradigm is to learn a reconstruction function $\feit:\dot{\V } \mapsto \dot{\gamma}$ using many training data $\{ (\dot{\V}^n, \dot{\gamma}^n): n=1, \ldots, N\}$.  The main issue is to find a suitable deep learning network ($\mathbb{DL}$) which allows to learn a useful reconstruction map $\feit$ from
 \begin{equation} \label{DeepNN0}
 \feit= \underset{f\in \mathbb{DL}}{\mbox{argmin}}\dfrac{1}{N}\sum_{n=1}^{N}\|f(\dot{\mathbf{V}}^n)-\dot{\gamma}^n\|^2.
\end{equation}
The deep learning-based reconstruction method exploits an integrated knowledge synthesis from the training data in order to get a direct reconstruction $\dot{\gamma}=\feit(\dot{\V})$ from a new measurement $\V$.

The deep learning method is very different from the conventional regularized model-fitting method, which can be expressed as:
  \begin{equation}\label{leastSquare0}
\dot{\gamma}=\underset{\dot{\gamma}\in \mathcal H}{\mbox{argmin}}  \| \dot{\V} - \Bbb S \dot{\gamma}\|^2 + \lambda\mbox{\it Reg}(\dot{\gamma}),
\end{equation}
where $\Bbb S\approx \f{\p\dot{\V}}{\p\dot{\gamma}}$ is the Jacobian matrix or sensitivity matrix (see Section 2.1 for details)  and $\mbox{\it Reg}(\dot{\gamma})$ is the regularization term enforcing the regularity of $\dot{\gamma}$, and $\lambda$ is the regularization parameter controlling the trade-off between the residual norm and regularity. Here, $\mathcal H$ is a space for representing images. In the case when the total number of pixels in the image is $d$, $\mathcal H=\R^d$. Since the dimension of $\mathcal H$ mostly is much bigger than the number of independent components in the measurement data $\dot{\V}$, a large number of possible images are consistent with the measurements up to the model and measurement error. Regularization is used to incorporate a-priori information in order to choose the image for which the regularization functional is smallest.
The success of this approach depends on whether the regularization term is indeed a good indicator for realistic lung images. In order to improve image
quality it seems desirable to go beyond standard regularization frameworks and add more specific a priori information.

In this work, we propose to use a deep learning method to find a useful constraint on EIT solutions for the lung ventilation model. We use a variational autoencoder learning technique (or manifold learning approach) to get a nonlinear expression of practically meaningful solutions $\dot{\gamma}$ by variables $\h$ in a low dimensional latent space, i.e., a decoder $\Psi$ is learned from training data to get the nonlinear representation $\dot{\gamma}=\Psi(\h)$. This generates the tripled training data $\{ (\dot{\V}^n, \dot{\gamma}^n, \h^n): n=1, \ldots N\}$. Next, we use the training data to learn a nonlinear regression map $\fVh: \dot{\V} \to \h$, which makes a connection between the latent variables $\h$ and the data $\dot{\V}$.  The nonlinear regression map $\fVh$ is obtained by
\begin{equation} \label{DeepNN-1}
\fVh= \underset{\fVh\in \mathbb{DL}_{h}}{\mbox{argmin}}
\dfrac{1}{N}\sum_{n=1}^{N}\|\fVh(\dot{\mathbf{V}}^n)-\h^n\|^2.
\end{equation}
where $\mathbb{DL}_{h}$ is a deep learning network described in section 2.3.  Then, the reconstruction map $\feit$ is expressed as $ \feit=\Psi \circ \fVh$.   The feasibility of the proposed method is demonstrated by using numerical simulations. The performance can be enhanced by accumulating high quality training data (clinically useful EIT images).

Before closing this introduction section, we should mention that our method does not use {\it ground-truth} labeled data for training, because lung EIT lacks a known ground truth at present. Although we have collected many human experiment data using 16 channel EIT system \cite{Kyounghuun2017}, its ground truthiness is not clear from a clinical point of view. Phantom experimental results cannot be used for ground-truth data, which are far from realistic. The collection of ground-truth training data may require a tough and complex process involving expensive clinical trials. The issue of collecting training data is beyond the scope of this paper.

\section{Time-difference EIT and conventional reconstruction methods}

\subsection{Time-difference EIT}
We briefly explain the mathematical model of an $E$-channel time-difference EIT system in which $E\in \mathbbm{N}$ electrodes are placed around the human thorax. See Fig. \ref{fig:human_thorax} for a sketch of a 16-channel EIT system. We assume that
measurements are taken in the following adjacent-adjacent pattern. A current of strength $I$ is driven
through the $j$-th pair of adjacent electrodes $(\mE^j,\mE^{j+1})$ keeping all other electrodes insulated, where we use the convention that $\mE^{E+1}=\mE^1$. Then the resulting electric potential $u^j_t$ satisfies approximately the shunt model equations (ignoring the contact impedances underneath the electrodes):
\begin{equation}\label{eq:shunt}
\left\{
\begin{array}{r@{\,}l}
\na\cdot(\gamma_t\na u^j_t)&=0 \quad \qquad \mbox{in}~~{\Om},\\
\int_{\mathcal{E}^j}\gamma\na u^j_t\cdot \n\,ds&=I=-\int_{\mathcal{E}^{j+1}}\gamma\na u^j_t\cdot \n \,ds,\\
(\gamma_t\na u^j_t)\cdot\n&=0 \quad \qquad \mbox{on}~~\p{\Om}\setminus \cup_i^{E}\mathcal{E}^i,\\
\int_{\mathcal{E}^i}\gamma_t \na u^j_t\cdot\n&=0 \quad \qquad \mbox{for}~~i\in\{1,\ldots,E\}\setminus\{j,j+1\},\\
u_t^j|_{\mE^i} &= \text{const.} \quad \mbox{for}~~ i=1,\ldots,E,\\
 \sum_{i=1}^{E} u_t^j|_{\mE^i} &=0,
\end{array}\right.
\end{equation}
where $\gamma_t$ is the conductivity distribution inside the imaging domain $\Omega$ at time $t$,
$\n$ is the outward unit normal vector to $\p{\Om}$, and $ds$ is the surface element.

Driving the current through the $j$-th pair of adjacent electrodes, we measure the voltage difference between the $k$-th pair of
adjacent electrodes
\[
V^{jk}_t=u_t^j|_{\mE^k}-u_t^j|_{\mE^{k+1}}.
\]
We measure $V^{jk}_t$ for all combinations of $j,k\in \{1,\ldots,E\}$ excluding
voltage measurements on current-driven electrodes since they are known to be highly affected by skin-electrode contact impedance
which is ignored in the shunt model \cite{harrach2015interpolation} for a possible remedy.
Thus the EIT measurements at time $t$ are given by the $E(E-3)$-dimensional vector
\[
\V_t=(V_t^{1,3},\ldots, V_t^{1,E-1}, V_t^{2,4}, \ldots,V_t^{2,E}, \ldots, V_t^{E,2},\ldots,V_t^{E,E-2})^T\in \R^{E(E-3)},
\]
where the superscript $T$ stands for the transpose of the vector.

\begin{figure}[h]
	\centering
	\includegraphics[width=1\textwidth]{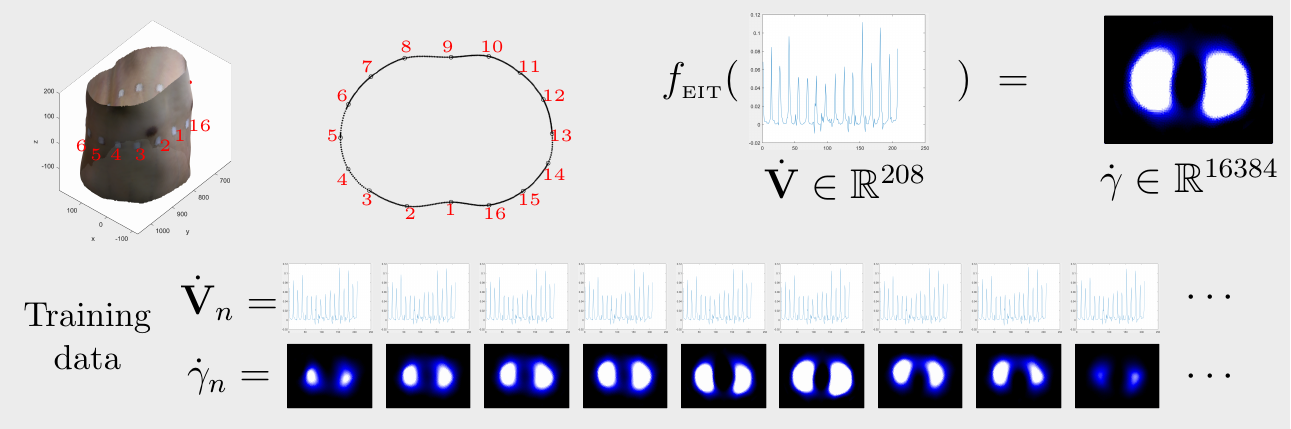}
	\caption{16 channel EIT system for monitoring regional lung ventilation. 16 electrodes are attached around the thorax to inject currents and measure boundary voltages.  The set of current-voltage data (i.e., a discrete version of Neumann-to-Dirichlet data) is used to reconstruct time-difference conductivity images. }
	\label{fig:human_thorax}
\end{figure}

In time-difference EIT, we use the difference of two measurements
\begin{equation}\label{NtDdata}
 \dot{\V}:=\V_{t}-\V_{t_0}\in \R^{E(E-3)},
\end{equation}
between sampling time $t$ and reference time $t_0$ in order to provide an image of the
conductivity difference
\[
\dot{\gamma}:= \gamma_{t}-\gamma_{t_0}.
\]
From the variational formulation of \eqref{eq:shunt}, one obtains the following linear approximation:
\begin{equation}\label{linear-approximation}
\begin{array}{l@{\,}l}
\dot{V}^{jk}&:=V^{jk}_t-V^{jk}_{t_0}=
\f{1}{I}\int_\Om \left[ \gamma_t\na u_t^j\cdot\na u_t^{k} - \gamma_{t_0}\na u_{t_0}^j\cdot\na u_{t_0}^{k} \right] d\r \\
&\approx \f{1}{I}\int_\Om  \dot{\gamma}\na u_{t_0}^j\cdot\na u_{t_0}^{k}  d\r.
\end{array}
\end{equation}
For a computerized image reconstruction, we discretize $\Om$ into finite elements $\Delta_m$, $m=1,2,\ldots, d$, as $\Om\approx \cup_{m=1}^d \Delta_m$, and assume that $\dot{\gamma}$ is approximately constant on each element $\Delta_m$.
Let $\dot\gamma_m\in \R$ denote the value of $\dot{\gamma}$ on $\Delta_m$ and identify $\dot{\gamma}$
with the column vector
\[
\dot{\gamma}=\left( \dot{\gamma}_1,\ldots,\dot{\gamma}_d \right)^T\in \R^d.
\]
Then \eqref{linear-approximation} can be written as
\[
\dot{V}^{jk}\approx \sum_{m=1}^d s_{jk}^m \dot{\gamma}_m \quad \text{ with } s_{jk}^m:=\f{1}{I}\int_{\Delta_m} \na u_{t_0}^j\cdot\na u_{t_0}^{k}  d\r.
\]
To write this in matrix-vector form, we fix $m=1,\ldots,d$, and write the elements of $s_{jk}^m$ as a $E(E-3)$-dimensional vector
\begin{equation}\label{sense_discrete}
\mathbf{S}^m=\left(s^m_{1,3},\ldots,s^m_{1,E-1}, s^m_{2,4}, \ldots,s^m_{2,E}, \ldots, s^m_{E,2}, \ldots,s^m_{E,E-2}\right)^T
\in \R^{E(E-3)}.
\end{equation}
Using these vectors as columns, we define the sensitivity matrix $\mathbb{S}\in \R^{E(E-3)\times d}$, and can thus write (\ref{linear-approximation}) as
\begin{equation}\label{eq:linearized}
\dot{\V}\approx
\sum_{m=1}^d \dot{\gamma}_m \mathbf{S}^m
=\left(
\begin{array}{ccc}
| &&| \\
\mathbf{S}^1& \cdots&\mathbf{S}^d\\
| &&|
\end{array}
\right)
\left(
\begin{array}{c}
\dot{\gamma}_1 \\
\vdots\\
\dot{\gamma}_d
\end{array}
\right)
= \mathbb{S}\, \dot{\gamma}.
\end{equation}

\begin{figure}[h]
	\centering
	\includegraphics[width=1\textwidth]{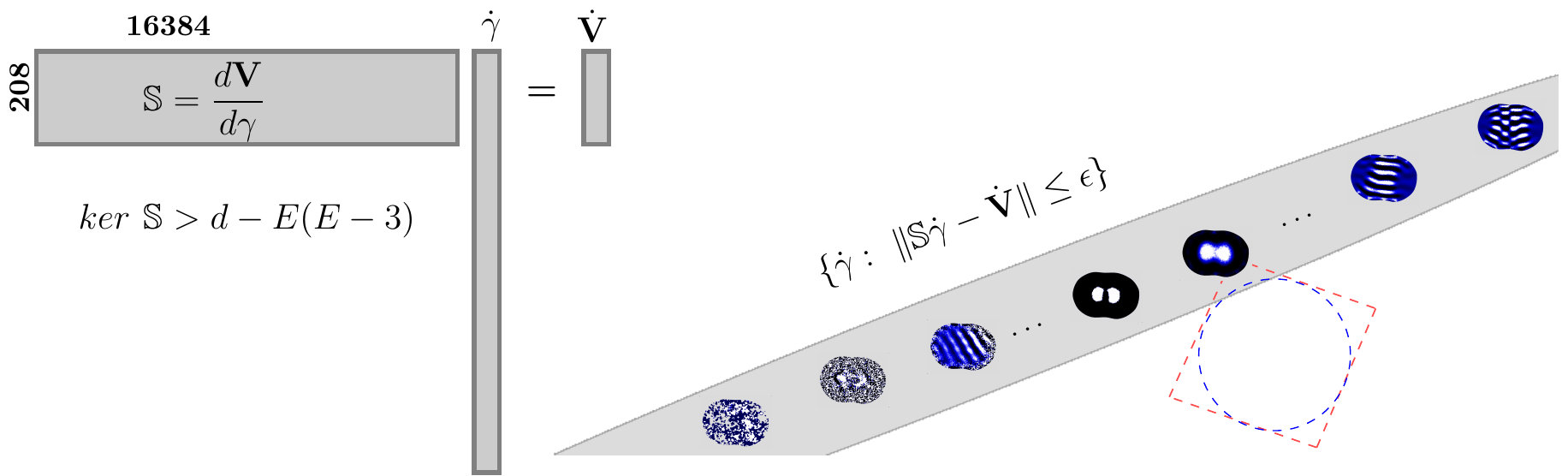}
	\caption{\cor{The system  $\Bbb S \dot{\gamma} =\dot{\V}$ on the left figure is a highly underdetermined problem. The set $\textsl{Sol}_{\ep}(\dot{\V})$ in (\ref{tolerance00}) can be viewed as a $\epsilon$-neighborhood of a space with dimension more than 16000.  The image on the right describes  conventional penalty-based regularization methods that selects an image from the set $\textsl{Sol}_{\ep}(\dot{\V})$.} }
	\label{fig:sensitivity}
\end{figure}
\subsection{Conventional penalty-based regularization methods}
In most practical applications, $d$ (the total number of pixels for $\dot{\gamma}$) is much bigger than $E(E-3)$ (the number of measurements), so that the linearized problem $\Bbb S \dot{\gamma}=\dot{\V}$ is a highly under-determined system.
When a 16-channel EIT system is used to produce images with $128\times 128$ pixels,
then the kernel dimension of $\Bbb S$ is at least $128^2-16*13$, so that a solution
of $\Bbb S \dot{\gamma}=\dot{\V}$ is only unique up to addition of an image coming from a
more than 16000-dimensional vector space.

Moreover, the linearized problem is only a rough approximation of the real situation and the measurements contain unavoidable noises.
Hence, all conductivity distributions $\dot\gamma$ in the wide region
\begin{equation}\label{tolerance00}
\textsl{Sol}_{\ep}(\dot{\V}):= \{ \dot{\gamma}\in \R^d ~:~ \| \Bbb S \dot{\gamma}-\dot{\V}  \|_\text{fid} \le  \ep \}
\end{equation}
have to be regarded as consistent with the measurements, where $\ep$ is a tolerance reflecting modeling and measurement errors,
and $\|\cdot\|_\text{fid}$ is a norm measuring the data fidelity. In the following we simply use $\|\cdot\|_2$ as fidelity norm.

Conventional penalty-based regularization methods reconstruct the conductivity image by choosing $\dot\gamma$ from all consistent candidates in $\textsl{Sol}_{\ep}(\dot{\V})$, so that it is smallest in some norm that penalizes unrealistic results.
Popular approaches include the simple Euclidean norm and the total variation norm which lead to the minimization problems
\begin{equation}\label{leastSquare00}
\dot{\gamma}=\underset{\dot{\gamma}}{\mbox{argmin}}  \| \dot{\V} - \Bbb S \dot{\gamma}\|_2^2 + \lambda\| \dot{\gamma}\|_2^2
\end{equation}
and
\begin{equation}\label{leastSquare1}
\dot{\gamma}=\underset{\dot{\gamma}}{\mbox{argmin}}  \| \dot{\V} - \Bbb S \dot{\gamma}\|^2 + \lambda\| D \dot{\gamma}\|_1,
\end{equation}
where $\lambda>0$ is a regularization parameter and $D\dot{\gamma}$ is the discretized gradient of $\dot{\gamma}$.

The performance of these approaches depends on whether the norm in the penalization term is indeed a good indicator
for realistic images. In order to improve image quality it seems desirable to add more
specific a priori information.

\subsection{Generic deep learning-based method}\label{subsect:generic_deep_learning}

Generic deep learning methods or lung monitoring in E-channel EIT system rely on a training data set of conductivity images and voltage measurements
\[
\{ (\dot{\gamma}_n, \dot{\mathbf{V}}_n)\in \Bbb R^{d}\times \Bbb R^{E (E-3)}:\ n=1,\ldots, N\},
\]
and aim to learn a useful
reconstruction map $\feit$ from a suitable class of functions described by a deep learning network $\mathbb{DL}$
through the minimization problem:
 \begin{equation} \label{DeepNN}
 \feit= \underset{f\in \mathbb{DL}}{\mbox{argmin}}\dfrac{1}{N}\sum_{n=1}^{N}\|f(\dot{\mathbf{V}}_n)-\dot{\gamma}_n\|^2.
\end{equation}
By using a training set, these methods incorporate very problem-specific a-priori information. But they do not take explicitly into
account that the EIT reconstruction problem is highly under-determined and ill-posed.


\section{A manifold learning based image reconstruction method}

\subsection{Motivation: Adding a manifold constraint}\label{subsect:paradigm}
To solve the highly under-determined system $\Bbb S \dot{\gamma}=\dot{\V}$, we follow the new paradigm
that realistic lung images lie on a non-linear manifold that is much lower dimensional than
the space of all possible images. If we can identify a suitable set $\mathcal M$ including images representing lung ventilation, then we can solve the constrained problem
\begin{equation}\label{LinearizedIP}
  \left\|
  \begin{array}{l}
  \mbox{Solve }~\Bbb S \dot{\gamma}~\approx~ \dot{\V}\\
  \mbox{subject to the constraint $\dot{\gamma}\in \mathcal M$. }  \end{array}
  \right.
\end{equation}
The unknown constraint $\mathcal M$ is hoped to be a low dimensional manifold of images displaying lung ventilation such that the intersection $\mathcal M\cap \textsl{Sol}_{\ep}(\dot{\V})$ is non-empty and of small diameter.
With this $\mathcal M$, it is hoped that the constraint problem is ``approximately well-posed"
in the following approximate version of the Hadamard well-posedness \cite{Hadamard1902}:
\begin{enumerate}
  \item[(a)] (Approximate uniqueness and stability) If two images $\dot{\gamma}, \dot{\gamma}'\in \mathcal M$ satisfy $\Bbb S\dot{\gamma}\approx \Bbb S\dot{\gamma}'$, then $\dot{\gamma}\approx \dot{\gamma}'$.
      \item[(b)] (Approximate existence) For any lung EIT data $\dot{\mathbf{V}}$, there exist $\dot{\gamma}\in \mathcal M$ such that $\Bbb S\dot{\gamma}\approx \dot{\mathbf{V}}$.
 \end{enumerate}
Many new theoretical and practical problems arise with this new paradigm.
It is a highly challenging question how to identify and describe manifolds displaying lung ventilation
on which the constrained inverse problem (\ref{LinearizedIP}) is robustly solvable.
A recent step in this direction is the result in \cite{harrach2018uniqueness} which shows that the inverse problems of EIT with
sufficiently many electrodes is uniquely solvable and Lipschitz stable on finite dimensional linear subsets of piecewise-analytic functions.

\subsection{Well-posedness of the inverse conductivity problem on compact sets}
The average image $\f{\dot{\gamma}^1+\dot{\gamma}^2}{2}$ of two different images $\dot{\gamma}^1$ and  $\dot{\gamma}^2$ displaying lung ventilation may not be a useful representation of lung ventilation. Hence, it is desirable to work with low dimensional non-linear manifolds for the conductivity image rather than with low dimensional vector spaces. As a first result to show that the inverse conductivity can be approximately well-posed
under non-linear constraints, we will now show that the inverse conductivity problem with continuous measurements (modeled by the Neumann-Dirichlet-operator) uniformly continuously determines the conductivity in compact sets of piecewise analytic functions.
We expect that the result also holds for
voltage measurements on a sufficiently high number of electrodes though that would require results on the approximation of the Neumann-Dirichlet-operator with the shunt electrode model that are outside the scope of this work.

\cor{For the following result let us also stress that the unique solvability of the inverse conductivity problem for
piecewise analytic conductivity functions and the continuum model is a classical result from Kohn and Vogelius \cite{Kohn1984,kohn1985determining}.
Without further restriction, the inverse conductivity problem is highly ill-posed, and due to the non-linearity,
stability is not a trivial consequence of restricting the conductivity to compact subsets. Alessandrini and Vessella \cite{alessandrini2005lipschitz}
have proven Lipschitz stability for the continuum model when the conductivity belongs to an a-priori known bounded subset of a finite-dimensional linear subspace of $C^2$-functions,
and \cite{harrach2018uniqueness} shows Lipschitz stability for bounded subsets of finite-dimensional linear subspace of piecewise-analytic functions
for the complete electrode model with sufficiently many electrodes. The following result follows the ideas from \cite{harrach2018global,harrach2018uniqueness} (see also \cite{harrach2019monotonicity}) to show that stability holds on any (possibly non-linear) compact subset of piecewise-analytic functions. 
It indicates that our new approach of constructing a low-dimensional non-linear manifold of useful lung images may indeed convert the ill-posed problem into a well-posed one.
}

\begin{theorem}
Let $C\subseteq L^\infty_+(\Omega)$ be a compact set of piecewise analytic functions (in the sense of \cite{kohn1985determining}).
For $\gamma\in C$ let $\Lambda(\gamma)$ denote the Neumann-Dirichlet-operator, i.e.,
\[
\Lambda(\gamma):\ L_\diamond^2(\partial \Omega)\to L_\diamond^2(\partial \Omega),\quad g\mapsto u|_{\partial \Omega},
\]
where $u\in H_\diamond^1(\Omega)$ solves $\nabla\cdot (\gamma\nabla u)=0$ in $\Omega$.

Then for all $\epsilon>0$ there exists $\delta>0$ so that for all $\gamma_1,\gamma_2\in C$
\[
\| \Lambda(\gamma_1)-\Lambda(\gamma_2) \|< \delta \quad \text{ implies } \quad \| \gamma_1-\gamma_2 \|< \epsilon.
\]
\end{theorem}
\begin{proof}
Let $\epsilon>0$. As in \cite{harrach2018uniqueness}, we have that for all $\gamma_1,\gamma_2\in C$ with $\|\gamma_1-\gamma_2\|\geq \epsilon$
\begin{align*}
\|\Lambda(\gamma_1)-\Lambda(\gamma_2)\| & \geq \sup_{g\in L_\diamond^2(\partial\Omega),\ \|g\|=1} f(\gamma_1,\gamma_2,\gamma_2-\gamma_1,g)\\
&\geq \inf_{\tau_1,\tau_2\in C,\ \kappa\in K} \sup_{g\in L_\diamond^2(\partial\Omega),\ \|g\|=1} f(\tau_1,\tau_2,\kappa,g),
\end{align*}
\cor{where $f$ is defined by taking the maximum of two values arising from monotonicity inequalities in \cite{harrach2018uniqueness}
\begin{align*}
f&:\  L^\infty_+(\Omega)\times L^\infty_+(\Omega) \times L^\infty(\Omega)\times L_\diamond^2(\partial \Omega) \to \R,\\
f(\tau_1,\tau_2,\kappa,g)&:=\max\{ \langle (\Lambda'(\tau_1)\kappa)g,g\rangle, -\langle (\Lambda'(\tau_2)\kappa)g,g\rangle\},
\end{align*}
}and
\begin{align*}
K&:=\{ \kappa=\tau_1-\tau_2:\ \tau_1,\tau_2\in C,\ \|\tau_1-\tau_2\|\geq \epsilon\}.
\end{align*}
From the compactness of $C$, it easily follows that also $K$ is compact. The function
\[
(\tau_1,\tau_2,\kappa)\mapsto \sup_{g\in L_\diamond^2(\partial\Omega),\ \|g\|=1} f(\tau_1,\tau_2,\kappa,g)
\]
is lower semicontinuous (see \cite{harrach2018uniqueness}) and thus attains its minimum over the compact set
$C\times C\times K$. With the same arguments as in \cite[Lemma~2.11]{harrach2018uniqueness} it follows that
\[
\sup_{g\in L_\diamond^2(\partial\Omega),\ \|g\|=1} f(\tau_1,\tau_2,\kappa,g)>0 \quad \text{ for all } (\tau_1,\tau_2,\kappa)\in C\times C\times K,
\]
so that we obtain
\begin{align*}
\delta:=&\inf_{\tau_1,\tau_2\in C,\ \kappa\in K} \sup_{g\in L_\diamond^2(\partial\Omega),\ \|g\|=1} f(\tau_1,\tau_2,\kappa,g) > 0.
\end{align*}
Hence,
\[
\|\Lambda(\gamma_1)-\Lambda(\gamma_2)\| \geq \delta \quad \text{ for all } \quad \gamma_1,\gamma_2\in C \text{ with } \|\gamma_1-\gamma_2\|\geq \epsilon,
\]
so that the assertion follows by contraposition.
\end{proof}

\subsection{Filtered data}\label{filter} Before we aim to find a manifold representation of lung ventilation images,
we preprocess the voltage measurements to remove geometry modeling errors.
In practical lung EIT, it is cumbersome to take account of patient-to-patient variability in terms of the boundary geometry and electrode positions, and it requires considerable effort to accurately estimate geometry information. Moreover, the voltage measurements $\dot{\V}$ can be affected by respiratory motion artifacts. Hence, it is desirable to filter out these boundary uncertainties as much as possible, to extract a ventilation-related signal, denoted by $\dot{\V}_{\mbox{\footnotesize lung}}$.

To this end, we preprocess the voltage measurements as in \cite{Kyounghuun2017}. We extract the boundary error, denoted by $\dot{\mathbf{V}}_{\mbox{\tiny err}}$,  by using the boundary sensitive Jacobian matrix $\mathbb{S}_{\mbox{\tiny bdry}}$:
$$
\dot{\mathbf{V}}_{\mbox{\tiny err}} :=\mathbb{S}_{\mbox{\tiny bdry}}\left(\mathbb{S}^T_{\mbox{\tiny bdry}}\mathbb{S}_{\mbox{\tiny bdry}}+
\lambda\mathbb{I}\right)^{-1}\mathbb{S}^T_{\mbox{\tiny bdry}}\dot{\mathbf{V}}
$$
where $\lambda$ is a regularization parameter, $\mathbb{I}$ is the identity matrix, and $\mathbb{S}_{\mbox{\tiny bdry}}$ is a sub-matrix  of  $\mathbb{S}$ consisting of all columns corresponding to the triangular elements located adjacent to the boundary. Then, the filtered data $\dot{\V}_{\mbox{\footnotesize lung}}= \dot{\mathbf{V}}- \dot{\mathbf{V}}_{\mbox{\tiny err}}$ is not so sensitive to the boundary $\p\Om$ and motion artifacts \cite{Kyounghuun2017}.

From now on, we use this filtered data $\dot{\V}_{\mbox{\footnotesize lung}}$ for the reconstruction instead of  $\dot{\V}$, in order to alleviate the boundary error and motion artifacts.  For notational simplicity, we use the same notation $\dot{\mathbf{V}}$ for the filtered data $\dot{\V}_{\mbox{\footnotesize lung}}$.

\subsection{Low dimensional manifold representation}\label{subsect:low_dim_manifold}

Assume that we are given a training data set of conductivity images and voltage measurements from an $E$-channel EIT system
\[
\{ (\dot{\gamma}_n, \dot{\mathbf{V}}_n)\in \Bbb R^{d}\times \Bbb R^{E (E-3)}:\ n=1,\ldots, N\}.
\]
Instead of directly applying a generic deep learning approach as described in subsection \ref{subsect:generic_deep_learning}, we follow
the new paradigm described in subsection \ref{subsect:paradigm} that images of lung ventilation
lie on a low dimensional manifold $\mathcal M$ on which the inverse problem is approximately well-posed.

We therefore first use the conductivity images in the training data set
\[
\{ \dot{\gamma}_n\in \Bbb R^{d}:\ n=1,\ldots, N\}
\]
to generate the low dimensional manifold $\mathcal{M}$.
In an $E$-channel EIT system, the number of independent information of current-voltage data is at most $E(E-3)/2$, due to the reciprocity $\dot{V}^{ji}=\dot{V}^{ij}$. Hence, in order to make the inverse problem approximately well-posed, we aim to generate $\mathcal M$ with dimension less than $E(E-3)/2$.
\begin{figure}[h!]
	\centering
	{\includegraphics[width=0.7\textwidth]{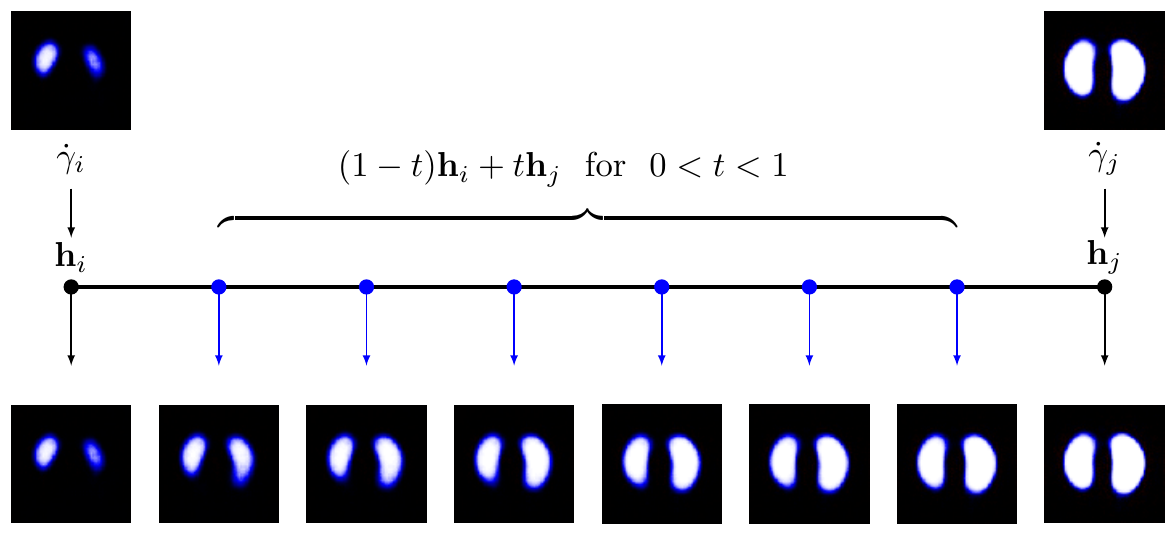}}
	\caption{\cor{Interpolation between two points $\h_i$ and $\h_j$ in the latent space. Given two images $\dot \gamma_i=\Psi(\h_i)$ and $\dot \gamma_j=\Psi(\h_j)$, VAE allows to generate  the interpolated image $\Psi((1-t)\h_i+t\h_j)$ for $0<t<1$. }}
	\label{fig:result2}
\end{figure}

\begin{figure}[h!]
	\centering
	{\includegraphics[width=0.9\textwidth]{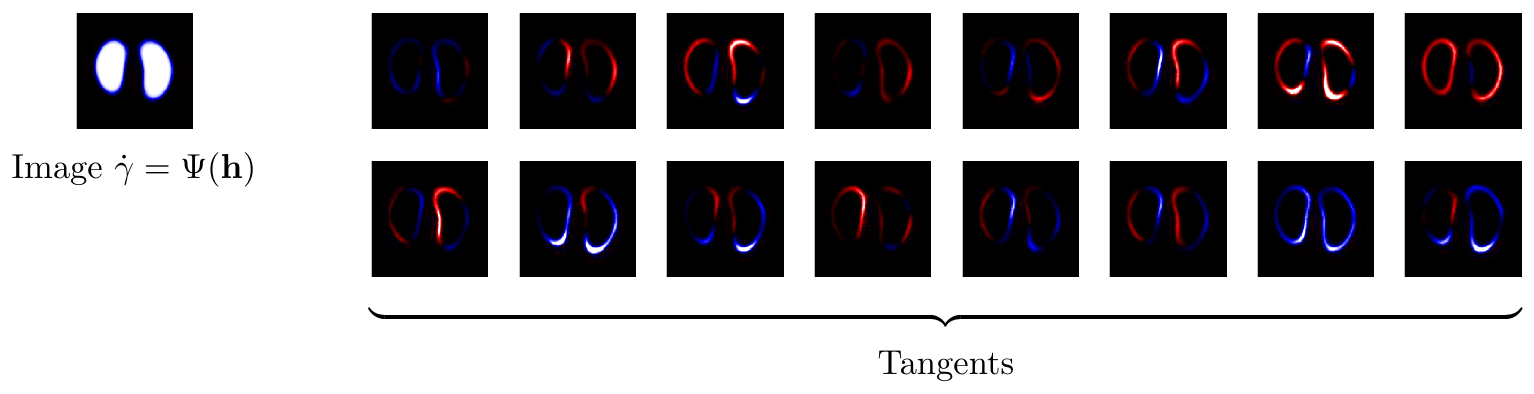}}
	\caption{\cor{Tangent vector of $\mathcal M$. Assuming that $\dot{\gamma}=\Psi(\h)$ is the image on the top left, its gradient $\nabla \Psi(\h)$ can be expressed as the images on the right side. }}
	\label{fig:tangent}
\end{figure}
\subsubsection{Autoencoder}
\cor{Given a dataset of lung EIT images, a variational autoencoder (VAE) \cite{Kingma2013} technique is used to learn the distribution of lung EIT images with the assumption that lung EIT image data (high dimensional) actually lies on a low dimensional manifold $\mathcal M$. }  For the ease of explanation of our idea, we start by first explaining the proposed method with the well-known standard autoencoder, instead of VAE.
Autoencoder uses the training data set $\{ \dot{\gamma}_n\in \Bbb R^{d}:\ n=1,\ldots, N\}$ to learn two functions
(called encoder and decoder)
\[
\Phi:\ \R^d\to \R^k \quad \text{ and } \quad \Psi:\ \R^k\to \R^d
\]
from a class of functions $\mathbb{AE}$ described by a deep learning network by minimizing
\begin{equation} \label{DeepNN2}
(\Psi,\Phi)=  \underset{(\Psi,\Phi) \in \mathbb{AE}}{\mbox{argmin}}\dfrac{1}{N}\sum_{n=1}^{N} \| \Psi\circ\Phi (\dot{\gamma}_n) -\dot{\gamma}_n\|^2.
\end{equation}
Choosing $k<<d$, one can interpret the encoder's output $\h=\Phi(\dot{\gamma})$ as a compressed latent representation, whose dimensionality is much less than the original size of the image $\dot{\gamma}$. The decoder $\Psi$ converts $\h$ to an image similar to the original input
\begin{equation} \label{PsiPhi}
\Psi\circ\Phi(\dot{\gamma}) \approx \dot{\gamma}.
\end{equation}
For our application of lung imaging using an $E$-channel EIT system, we choose the class of functions $\mathbb{AE}$ to contain encoder functions $\Phi$ of the form
\begin{equation} \label{encoder}
\Phi(\dot{\gamma}):=W^{\ell-1}\circledast
 \left(\eta (W^{\ell-2}\circledast \eta
 \left(\cdots
 \eta\left(W^1\circledast \dot{\gamma}\right)
 \cdots \right)
 \right)
\end{equation}
and decoder functions $\Psi$ of the form:
\begin{equation} \label{decoder}
 \Psi(\h)=\tanh\left(W^{2\ell}\circledast^\dag
 \left(\eta (W^{2\ell-1}\circledast^\dag  \eta
 \left(\cdots
 \eta\left(W^{\ell+1}\circledast^\dag \h\right)
 \cdots \right)
\right)
\right)
\end{equation}
Here,  $W\circledast \x$ and $W\circledast^\dag\x$, respectively, are  the convolution and transposed convolution \cite{Zeiler2013} of $\x$ with weight $W$; $\tanh$ is the hyperbolic tangent function; $\eta$ is the rectified linear unit activation function $ReLU$.
The dimension $k$ (the number of the latent variables) is chosen to be smaller than $E(E-3)/2$ as motivated in subsection~\ref{subsect:low_dim_manifold}.  We hope that $\Phi$ and $\Psi$  satisfy:
\begin{enumerate}
\item[(P1)] $\Psi(\Phi(\dot{\gamma}_n))\approx \dot{\gamma}_n$, i.e., the lung ventilation conductivity images in the training data set approximately lie on the low dimensional manifold
\[
\mathcal M=\{ \Psi(\h): \h \in \R^k\}.
\]
\item[(P2)] $\mathcal M$ is a manifold of useful lung EIT images. \cor{In particular, this means that for two images $\dot \gamma_i=\Psi(\h_i)$ and $\dot \gamma_j=\Psi(\h_j)$ in $\mathcal M$, the interpolated image $\Psi((1-t)\h_i+t\h_j)$ should represents a lung EIT image between $\dot \gamma_i$ and $\dot \gamma_j$.  }
\end{enumerate}
Definitely, the autoencoder approach aims to fulfill (P1) by minimizing the reconstruction loss of (\ref{DeepNN2}).
However,  \cor{the second property (P2), as shown in Fig. \ref{fig:result2}, may not be satisfied by the classical deterministic autoencoder approach (\ref{DeepNN2}).   There may be holes in the latent space on which the decoder is never trained \cite{Rubenstein2018}. Hence, $\Psi((1-t)\h_i+t\h_j)$ for some $t$ may be an unrealistic lung ventilation image. This is the reason why we use variational autoencoder, which can be viewed as a regularized autoencoder or nonlinear principal component analysis\cite{Bengio2013,Kingma2013}.}

\cor{Let us also stress, that the mappings $\Psi$ and $\Phi$ will only be approximately inverse to each other, so that $\mathcal M$
might not be a manifold in the strict mathematical sense. However, the set $\mathcal M$ constructed by this approach (also including the VAE-approach described in the next subsection) will always be an image of the low-dimensional latent space $\R^k$ under the continuous mapping $\Psi$.
For the sake of readability, we keep the somewhat sloppy terminology and refer to $\mathcal M$ as low-dimensional manifold.
Moreover, note that the image of $\Psi$ of a closed bounded subset of the latent space $\R^k$ will be compact.}

\subsubsection{Variational autoencoder (VAE)}

The idea of VAE is to add variations in the latent space to the minimization problem (\ref{DeepNN2}), in order to achieve (P2).
More precisely, in VAE, the encoder $\Phi$ is of the following nondeterministic form:
\begin{equation} \label{encoder2}
 \Phi(\dot{\gamma}) =\Phi_{\mbox{\tiny me}}(\dot{\gamma})+\Phi_{\mbox{\tiny std}}(\dot{\gamma})\odot \h_{\mbox{\tiny noise}}
\end{equation}
where $\Phi_{\mbox{\tiny me}}$ outputs a vector of means $\mu=(\mu(1), \cdots, \mu(k))\in\R^k$; $\Phi_{\mbox{\tiny std}}$ outputs a vector of standard deviation $\sigma=(\sigma(1),\cdots, \sigma(k))\in \R^k$; $ \h_{\mbox{\tiny noise}}$ is an auxiliary noise variable sampled from standard normal distribution $\mathcal{N}(0,I)$; and $\odot$ is the element-wise product (Hadamard product). Here, $\Phi_{\mbox{\tiny me}}$ and $\Phi_{\mbox{\tiny std}}$  are of the form (\ref{encoder}) and describe the mean vector $\mu= \Phi_{\mbox{\tiny me}}(\dot{\gamma})$ and the standard variation vector $\sigma=\Phi_{\mbox{\tiny std}}(\dot{\gamma})$ of the non-deterministic encoder function.

\cor{According to (\ref{encoder2}), 
\[
\Phi(\dot{\gamma})=\h \sim \mathcal{N}({\mu},{\Sigma}),
\]
where $\Sigma$ is a diagonal covariance matrix $\Sigma=\mbox{diag}(\sigma(1)^2, \cdots, \sigma(k)^2)$.}
With this non-deterministic approach, we can fulfill the property (P2) since the same input $\dot \gamma$ can now be encoded as a whole range of perturbations of $\h$ in the latent space, and thus we can determine a decoder function $\Psi$ that
maps a whole range of perturbations of $\h$ to useful lung images. \cor{To find $\Psi$,
note that for all images $\dot{\gamma}$, the concatenation $\Psi(\Phi(\dot{\gamma}))$ is now a random vector.
Since we can also interpret $\dot{\gamma}$ as a random vector which always takes the same value, we could ensure the desired property (P1)
by simply minimizing (\ref{DeepNN2}) with $ \| \cdot \|$ now denoting the energy distance between two random vectors. But this trivial approach would obviously still prefer a deterministic
encoder, i.e., $\Phi_{\mbox{\tiny me}}$ will be the encoder function from the standard autoencoder approach, and $\Phi_{\mbox{\tiny std}}\equiv 0$.}

\cor{Hence, in order to ensure variations in the latent space to achieve (P2), we additionally enforce that the distribution of the encoder output is
close to a normal distribution. We thus minimize (\ref{DeepNN2}) 
we minimize (\ref{DeepNN2})
with an additional term that penalizes the Kullback-Leibler (KL) divergence loss between $\mathcal{N}(\mu_n,\Sigma_n )$ and $\mathcal{N}(0,I))$ for all $n=1, \cdots, N$}
\begin{equation} \label{KL2}
D_{KL}(\mathcal{N}(\mu_n,\Sigma_n ) \parallel \mathcal{N}(0,I))= \frac{1}{2} \sum _{j=1}^{k} \left[(\mu_n(j)^2+ \sigma_n(j)^2- \log \sigma_n(j) -1\right].
\end{equation}
We thus obtain the VAE method
\begin{equation}\label{DeepNN3}
(\Psi,\Phi)=  \underset{(\Psi,\Phi) \in \mathbb{VAE}}{\mbox{argmin}}\dfrac{1}{N}\sum_{n=1}^{N}\left[  \| \Psi\circ\Phi (\dot{\gamma}_n) -\dot{\gamma}_n\|^2
+D_{KL}(\mathcal{N}(\mu_n,\Sigma_n ) \parallel \mathcal{N}(0,I))   \right]
\end{equation}
where $\mu_n= \Phi_{\mbox{\tiny me}}(\dot{\gamma}_n)$ and $\sigma_n=\Phi_{\mbox{\tiny std}}(\dot{\gamma}_n)$.
We should note that the covariance $\Sigma_n$ and the term
$
D_{KL}(\mathcal{N}(\mu_n,\Sigma_n ) \parallel \mathcal{N}(0,I))
$
allows smooth interpolation and compactly encoding, resulting in generating compact smooth manifold.

\subsection{The image reconstruction algorithm}
Now, we are ready to explain the reconstruction algorithm $\feit$. Given a set of training data,
the key idea is that we do not aim to learn a nonlinear regression map that directly
reconstructs the conductivity $\dot\gamma$ from the voltage measurements $\dot{\V}$ as this will be a highly
under-determined and ill-posed problem. Instead we first use the variational autoencoder method as explained in the last subsection
to identify a low dimensional latent space encoding the manifold of useful lung images, and then learn the nonlinear regression map
that reconstructs the low-dimensional latent variable as this problem can be expected to be considerably better posed.

To explain this in more detail, let $\{ (\dot{\V}_n, \dot{\gamma}_n):\ n=1, \ldots, N\}$ be a set of training data.
Using the learned encoder  $\Phi_{\mbox{\tiny me}}$ in (\ref{encoder2}), we obtain a set of training data for the latent variable $\{ (\dot{\V}_n, \h_n): n=1, \ldots, N\}$ with
\begin{equation} \label{meanhn}\h_n:=\Phi_{\mbox{\tiny me}}(\dot{\gamma}_n).
\end{equation}
In order to learn a nonlinear reconstruction map that reconstructs the latent variable from the voltage measurements, i.e.
\begin{equation} \label{DeepNN-11}
\fVh( \dot{\V})\approx \h,
\end{equation}
we minimize
\begin{equation} \label{DeepNN-10}
\fVh= \underset{\fVh\in \mathbb{DL}_{h}}{\mbox{argmin}}
\dfrac{1}{N}\sum_{n=1}^{N}\|\fVh(\dot{\mathbf{V}}_n)-\h_n\|^2
\end{equation}
where $\mathbb{DL}_{h}$ is the multilayer perceptrons with their mathematical representation given by
\begin{equation}\label{fVh}
 \fVh^*(\dot{\V}) = W_{\tiny \sharp}^{\ell-1}
 \left(\eta \left(W_{\tiny \sharp}^{\ell-2}
 \left(\cdots
\eta \left(W_{\tiny \sharp}^1 \dot{\V}\right)
 \cdots \right)
 \right)
 \right),
\end{equation}
where $W_{\tiny \sharp} \x$ is the matrix multiplication of $\x$ with weight $W_{\tiny \sharp}$ and $\eta$ is $ReLU$. See Fig. \ref{fig:EIT_scheme} for details.

After finding $\fVh$ by solving the minimization problem (\ref{DeepNN-10}), we can reconstruct the conductivity from the latent variable
by applying the decoder $\Psi$ in (\ref{decoder}).  In summery, the proposed lung EIT reconstruction map is:
\begin{equation}\label{recon}
\feit:=\Psi\circ\fVh:\ \dot{\V}~\longrightarrow~ \h  ~\longrightarrow~ \dot{\gamma} .
\end{equation}

\begin{figure}[h]
	\centering
	\includegraphics[width=0.9\textwidth]{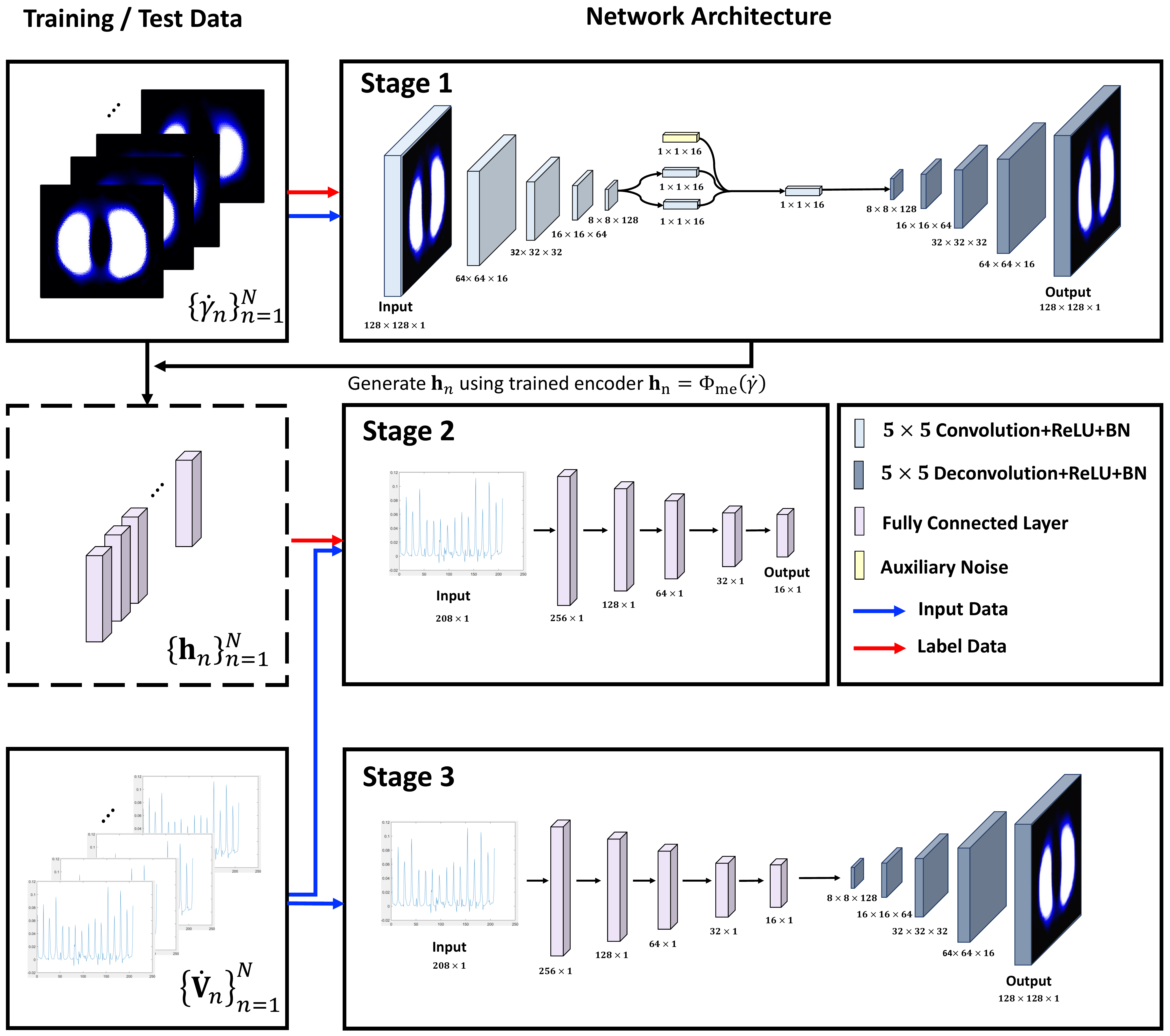}
	\caption{Architecture of the proposed image reconstruction method. In the first stage, variational autoencoder is used to learn a 16-dimensional manifold representation for getting prior knowledge of lung EIT images. In the second stage, a map $\fVh:\dot{\V}\rightarrow \h$ is trained with $\{(\dot{\V}_n,\h_n)\}_{n=1}^{N}$. Here $\h_n$ were given by encoder; $\Phi_{\mbox{\tiny me}}(\dot{\gamma}_n)$. In the third stage, $\feit:=\Psi\circ\fVh^*$ works with trained $\fVh^*$ and $\Psi$.}
	\label{fig:EIT_scheme}
\end{figure}

\section{Experiments and Results}

\subsection{Generating labeled data}

We numerically generate a set of labeled data $\{ (\dot{\V}_n, \dot{\gamma}_n): n=1, \cdots N\}$ using the forward model (\ref{eq:shunt}) with 16-channel EIT system and the filtering process in section \ref{filter}. To mimic practical situations, we use some human experiment results by the fidelity-embedded reconstruction method \cite{Kyounghuun2017} to collect a set of labeled data $\{ (\dot{\V}_m, \dot{\gamma}_m): m=1, \cdots, k\}$. We also interpolate these data to generate an additional data  by computing  the forward problem (\ref{eq:shunt}) and (\ref{NtDdata}). The number of training data  $\{ (\dot{\V}_n, \dot{\gamma}_n): n=1, \cdots N\}$ was 21360. For data augmentation purpose, we added 10 different 5\% Gaussian random noise to $\dot{\V}$. The size of images $ \dot{\gamma}_n$ is $128\times128$. All training was performed using an NVIDIA GeForce GTX 1080ti GPU.

\subsection{Training procedure and reconstruction result\label{alg:EIT_alg1}}
The proposed method consists of three stages: (i) Training variational autoencoder to find a low-dimensional representation; (ii) Training the nonlinear regression map $\fVh$ from EIT data to latent variables; (iii) EIT Image Reconstruction.

\begin{algorithm}[h]
\caption{\small The proposed training and reconstruction algorithm.}
\begin{enumerate}
\setlength{\itemindent}{2em}
  \item [\textbf{Stage 1.}] \textbf{Training variational autoencoder to find a low-dimensional representation}

    \begin{algorithmic}
    \FOR{number of training step}
        \STATE{$\bullet$ Sample the minibatch of $m$ image $\{\mathbf{\dot{\gamma}}_{1},\cdots,\mathbf{\dot{\gamma}}_{m}\}$ from training data.}
        \STATE{$\bullet$ Sample the minibatch of $m$ auxiliary noise $\{\h_{\mbox{\tiny noise},1},\cdots,\h_{\mbox{\tiny noise},m}\}$ from standard normal $\mathcal{N}(\mathbf{0},\mathbf{I} )$.}
        \STATE{$\bullet$ Update the parameters of VAE using the gradient of the loss  $\mathcal L_{\tiny{\mbox{1}}}$ in (\ref{DeepNN3}) with respect to the parameters of VAEs for the minibatch:
       {\small $$
	\mathcal L_{\tiny{\mbox{1}}}  = \dfrac{1}{m}\sum_{n=1}^{m}\left[  \| \Psi\circ\Phi (\dot{\gamma}_n) -\dot{\gamma}_n\|^2
+D_{KL}(\mathcal{N}(\mu_n,\Sigma_n ) \parallel \mathcal{N}(0,I))   \right]
$$}}
    \ENDFOR
    \end{algorithmic}
  \item [\textbf{Stage 2.}] \textbf{Training the nonlinear regression map $\fVh$}

  \begin{algorithmic}
    \FOR{number of training step}
        \STATE{$\bullet$ Sample the minibatch of $m$ image $\{\mathbf{\dot{\gamma}}_{1},\cdots,\mathbf{\dot{\gamma}}_{m}\}$ from training data and encode the sampled images to generate $\{ \Phi_{\mbox{\tiny me}}(\mathbf{\dot{\gamma}}_{1}) ,\cdots,\Phi_{\mbox{\tiny me}}(\mathbf{\dot{\gamma}}_{m}) \}$.}
        \STATE{$\bullet$  Sample the minibatch of $m$ paired voltage data $\{\mathbf{\dot{\V}}_{1},\cdots,\mathbf{\dot{\V}}_{m}\}$ from training data set.}
        \STATE{$\bullet$ Update the parameters of $\fVh$ using gradient of loss $\mathcal L_{\tiny{\mbox{2}}}$ in (\ref{DeepNN-10}) with respect to the parameters of $\fVh$ for the minibatch:
        {\small $$
	\mathcal L_{\tiny{\mbox{2}}}  = \dfrac{1}{m}\sum_{n=1}^{m}\|\fVh(\dot{\mathbf{V}}_n)-\h_n\|^2
$$}}
    \ENDFOR
  \end{algorithmic}
	\item [\textbf{Stage 3.}] \textbf{EIT Image Reconstruction}\\
  Using the trained nonlinear regression map $\fVh$ and decoder $\Psi$, a EIT reconstruction map $\feit$ is acheived by
  \begin{displaymath}
    \feit(\dot{\V})=\Psi\circ\fVh(\dot{\V}).
  \end{displaymath}
\end{enumerate}
\end{algorithm}

We used the AdamOptimizer \cite{Kingma2015} to minimize loss. The batch normalization \cite{Ioffe2015} was also applied.
After finishing the training process (stage1, 2), a EIT reconstruction images were given by $\feit(\dot{\V})=\Psi\circ\fVh(\dot{\V})$. The reconstruction result is shown in Fig. \ref{fig:result_recon}.

\begin{figure}[!h]
	\centering
    \begin{tabular}{c}\hline
        {\footnotesize{Reconstruction by proposed deep learning based method}}\\
        			{\includegraphics[width=0.9\textwidth ]{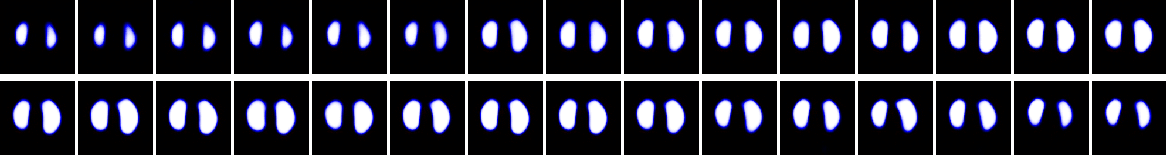}}\\
    \end{tabular}
    \caption{Reconstruction result of deep learning based method from real experimental data. The reconstruction were given by $\feit(\dot{\V})=\Psi\circ\fVh(\dot{\V})$.}
	\label{fig:result_recon}
\end{figure}

\subsection{Visualizations of learned manifold}

Our experimental result shows that lung EIT images lie on the low-dimensional smooth compact manifold. For easy visualization purpose, we visualized the lung EIT manifold with two dimensional latent space to project the high dimensional image to low dimensional manifold. Here we choose the equally spaced latent variables $\h\in[-3,3]^2$ and we decoded them to generate the images as shown in Fig. \ref{fig:latent_vis}.

\begin{figure}[ht!]
	\centering
	{\includegraphics[width=0.6\textwidth]{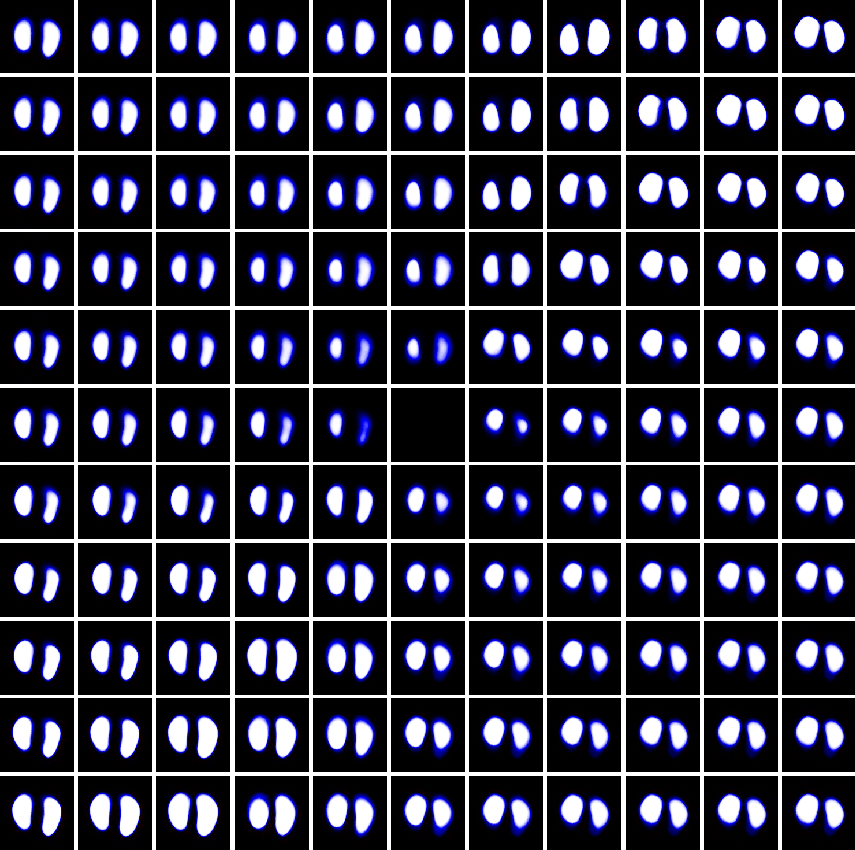}}
	\caption{Visualization of learned lung EIT manifold with two-dimensional latent space.}
	\label{fig:latent_vis}
\end{figure}

We also visualized manifold with 16-dimensional latent space. Since we can not directly visualize 16-dimensional manifold, we visualized the manifold along each axis of the 16-dimensional latent space as shown in Fig. \ref{fig:latent_pertub} (a). Here, each $i$-th row in Fig. \ref{fig:latent_vis} (a) shows lung EIT image $\Psi(\h_{i,j})$ with $\h_{i,j}=\delta_j \e_i$ where $\e_i$ is unit vector whose $i$-th component is one and otherwise zero with $\delta_j\in\{-6, \cdots, 0, \cdots, 6  \}$ for $i\in\{1,\cdots,16\}$ and $j\in\{1,\cdots,13\}$.
Each $(i,j)$ image in Fig. \ref{fig:latent_pertub} (b) shows the tangent which denotes the direction from $(i,j)$ image to $(i,j+1)$ image in Fig. \ref{fig:latent_pertub} (a) for $i\in\{1,\cdots,16\}$ and $j\in\{1,\cdots,12\}$. From manifold visualization, we can verify that change of lung images(e.g., lung ventilations) are observed when we walk in the latent space.

\begin{figure}[ht!]
	\centering
	\subfigure[]{\includegraphics[height=0.55\textwidth]{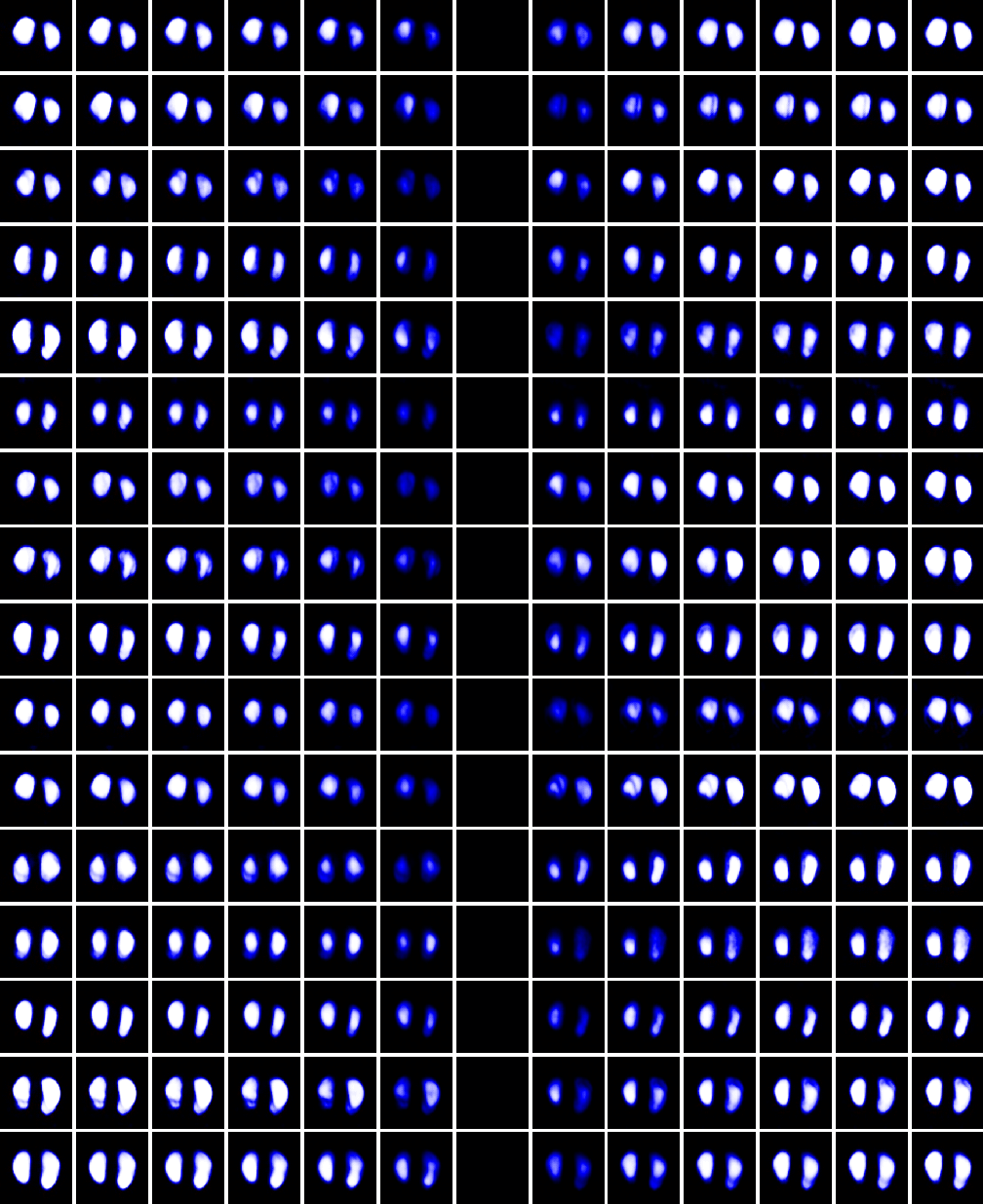}}
	\subfigure[]{\includegraphics[height=0.55\textwidth]{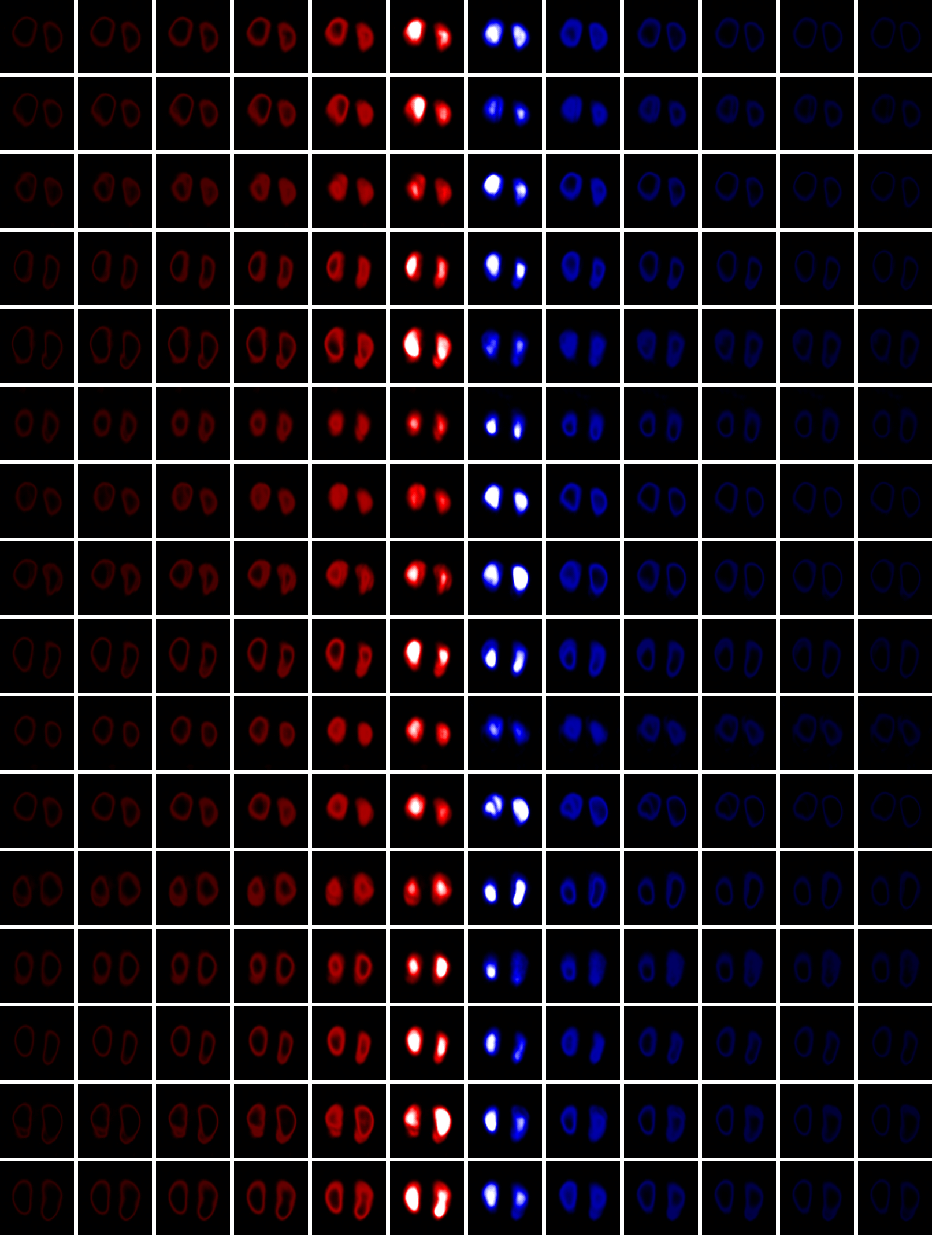}}
	\caption{Visualization of learned manifold with 16-dimensional latent space. (a) shows $\Psi(\h_{i,j})$ where $\h_{i,j}$ have value $\delta_j \in \{-6,\cdots,6\}$ in $i$-th component and otherwise zero. (b) shows the its tangents. We expected that small changes in latent space produce small changes in image space. Here each blue and red color denote the positive and negative value.}
	\label{fig:latent_pertub}
\end{figure}

\subsection{Advantages on VAE-based manifold constraint}

The proposed method is advantageous over the conventional regularization methods due to the low dimensional  manifold constraint in reconstructing lung images fitting EIT data.
The conventional methods does not work for obese people, which is the case where lung is placed away from the surface electrodes. The conventional regularization methods may produce merged images due to their fundamental nature penalizing image perturbation, as shown in Fig. \ref{fig:result_comparison}. On the other hand, the proposed method always generate lung-like images due to the learning constraint of lung images.

In this experiment, we use a simulated image $\dot\gamma$ and compute the corresponding simulated data $\dot{\V}$  using the forward model (\ref{eq:shunt}) with $\gamma=1+\dot\gamma$ and 16-channel EIT system. Here, we added 5\% Gaussian random noise to $\dot{\V}$. For each image of $\dot\gamma$, 10 different data are computed by adding the noise. Totally 21360(=2136$\times$ 10) data pairs are used for the training process. The images in Fig. \ref{fig:result_comparison} compares the proposed method with regularized data fitting methods $\underset{\dot{\gamma}}{\mbox{argmin}}\|\dot{\V}-\Bbb S\dot {\gamma}\|^2+\lambda \|\dot{\gamma}\|^2_2$ by using a simulated EIT data. In case2, as shown in Fig. \ref{fig:result_comparison}, two lungs are merged in the reconstructed images by the regularized data fitting methods, but not in the reconstructed image by the proposed method. \cor{It is because the measured data are highly sensitive to conductivity changes near the current-injection electrodes, whereas the sensitivity drops rapidly as the distance increases\cite{Barber1988}.}

\begin{figure}[!h]
	\centering
	{\includegraphics[width=0.9\textwidth]{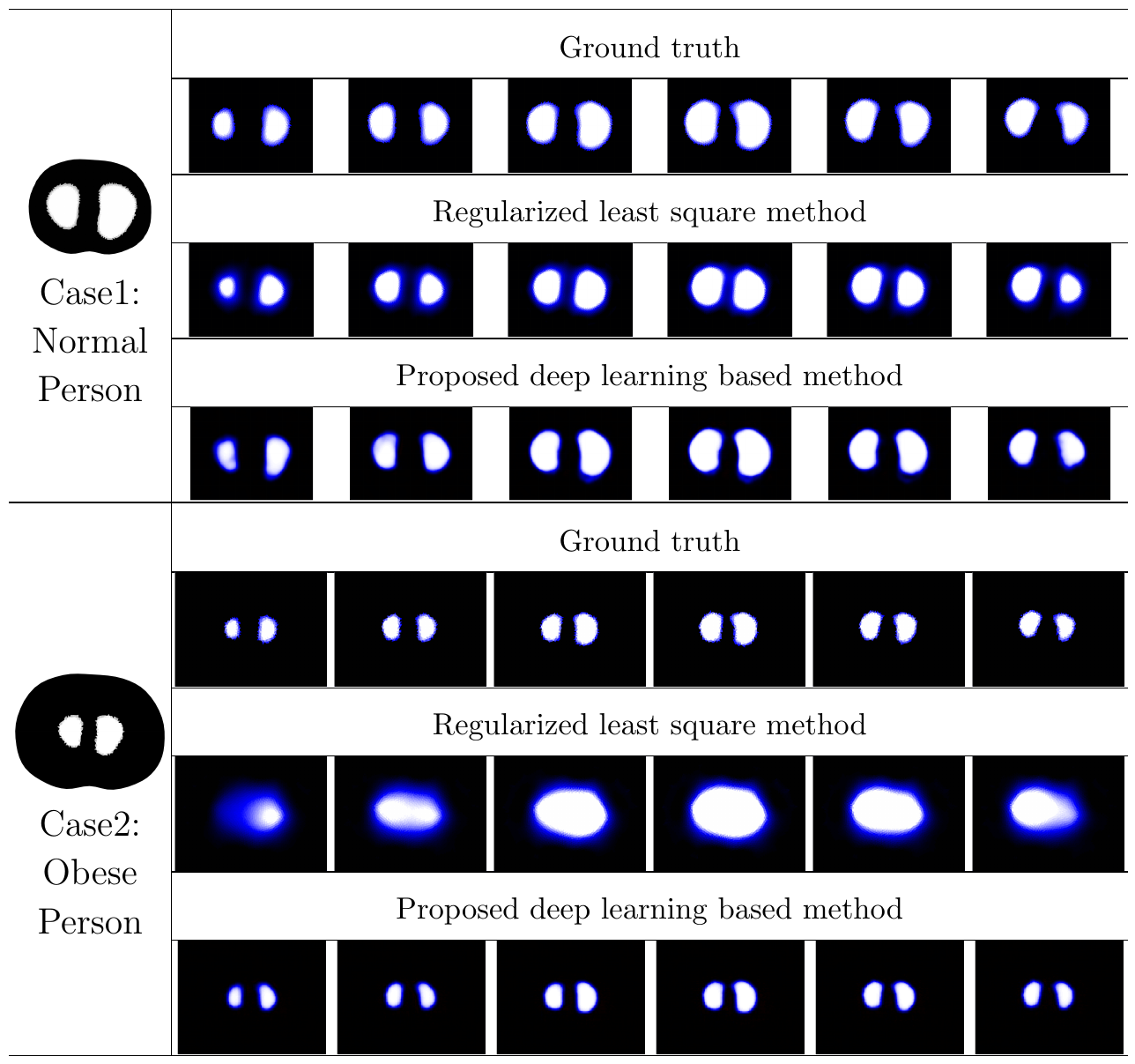}}
    \caption{Comparison of reconstruction methods with simulated data. We compare the proposed deep learning method  \ref{alg:EIT_alg1} with the standard regularized least square method \eref{leastSquare00} for two cases: normal person and obese person. In case 1 (normal person), both methods produces reasonably accurate reconstructions. In case 2 (obese person), however, the standard method gives merged image because electrodes positions are distant from the support of $\dot{\gamma}$. On the other hand, the proposed method provides useful reconstruction. Here, we resized our result of deep learning to the same ratio of reconstruction of conventional method for the comparison.  }
	\label{fig:result_comparison}
\end{figure}

\section{Discussion and Conclusion}
This paper addressed the problem of handling ill-posed nonlinear inverse problems by suggesting a low dimensional representation of target images.
ElT is a typical example of ill-posed nonlinear inverse problems where the dimension of measured data is much lower than the number of unknowns (pixels of the image). Moreover, there exist complicated nonlinear interrelations among inputs (a practical version of Dirichlet-to-Neuman data), outputs (impedance imaging), and system parameters. Finding a robust reconstruction map $\feit$ for clinical practice requires to use prior knowledge on image expression.   Regularization techniques have been used widely to deal with ill-posedness, but the conventional $L^p$-norm based regularization may not provide a proper prior of target images in practice. See Fig. \ref{fig:sensitivity}.

\cor{ Deep learning framework may provide a nonlinear regression on training data which acts as learning complex prior knowledge on the output. VAE allows to achieve compact representation (or low dimensional manifold learning) for prior information of lung EIT images, as shown in Fig. \ref{fig:result2} and  Fig. \ref{fig:tangent}. Dai {\it et al.} \cite{Dai2018} viewed VAE as the natural evolution of robust PCA models, capable of learning nonlinear manifolds of unknown dimension obscured by gross corruptions.
Given data $\{ \dot{\gamma}_n\in \Bbb R^{d}:\ n=1,\ldots, N\}$, the encoder $\Phi(\dot{\gamma})$ in \eqref{encoder2} can be viewed as a conditional distribution $q(\h|\dot{\gamma})$ that satisfies $q(\h|\dot{\gamma})=\mathcal{N}(\mu,\Sigma)$. The decoder $\Psi$ can be represented by a conditional distribution $p(\dot{\gamma}|\h)$ with $p(\h)=\mathcal{N}(0,I)$. VAE tries to match $p(\h|\dot{\gamma})$ and  $q(\h|\dot{\gamma})$.  VAE encoder covariance can help to smooth out undersirable minima  in the energy landscape of what would otherwise resemble a more traditional deterministic autoencoder \cite{Dai2018}.  }

Given the training data $\{ (\dot{\V}_n, \dot{\gamma}_n): n=1, \cdots, N\}$, the encoding-decoding pair $(\Phi, \Psi)$ and the nonlinear regression map $\fVh$ in (\ref{DeepNN-10}) satisfy the following properties as in the sense of Hadamard \cite{Hadamard1902}:
 \begin{itemize}
 \item {\it Approximate Existence :} Given $\dot{\V}$, there exist $\h$ such that $\fVh(\dot{\mathbf{V}})\approx \h$.
  \item {\it Approximate Uniqueness :} For any two different EIT data $\dot{\mathbf{V}}, \dot{\mathbf{V}}'$, we have
   $\| \Phi(\dot{\mathbf{V}}) -\Phi(\dot{\mathbf{V}}')\|  \gtrsim  \| \fVh(\dot{\mathbf{V}})-\fVh(\dot{\mathbf{V}}')\|$.
 \item {\it Stability :}  $\dot{\mathbf{V}}\thickapprox  \dot{\mathbf{V}}'$ implies  $\Phi(\dot{\mathbf{V}}) \thickapprox  \Phi(\dot{\mathbf{V}}')$
\end{itemize}

\cor{ The proposed deep learning approach is a completely different paradigm from regularized data-fitting approaches that use a “single” data-fidelity with regularization. The deep learning approach instead uses a “group” data fidelity to learn an inverse map from the training data. The deep learning framework can provide a nonlinear regression for the training data, which acts as learning complex prior knowledge of the output. Let us explain this using the well-known example of sub-Nyquist sampling (compressive sensing) MRI, which is an ill-posed inverse problem with fewer equations than unknowns. The well-known compressed sensing (CS) method with random sampling  is based on the regularized data-fitting approach (single data fidelity), where total variation regularization  is used to enforce the image sparsity to compensate for undersampled data \cite{Candes2006,Lustig2007}.  The CS method requires non-uniform random subsampling, since it is effective to reduce noise. On the other hand,  the deep learning-based method \cite{Hyun2018} provides a low-dimensional latent representation of MR images, which can be learned from the training set (group data fidelity). The learned reconstruction function from the group data fidelity appears to have highly expressive representation capturing anatomical geometry as well as small anomalies \cite{Hyun2018}.}

Deep learning techniques have expanded our ability by sophisticated ``disentangled representation learning'' though training data. DL methods appear to overcome limitations of existing mathematical methods in handling various ill-posed problems. Deep learning methods will improve their performance as training data and experience accumulate over time. However, we do not have rigorous mathematical grounds behind why deep learning methods work so well. We need to develop mathematical theories to ascertain their reliabilities.

\section*{Acknowledgement}
J.K.S. and K.C.K were supported by the National Research Foundation of Korea (NRF) grant 2015R1A5A1009350 and 2017R1A2B20005661.
K.L. and A.J. were supported by NRF grant 2017R1E1A1A03070653.

\stop
